\documentclass[a4paper,10pt,leqno]{amsart}
\usepackage{amsmath,amsthm,amssymb,amsfonts,enumerate,color,mathtools}
\usepackage[pdftex]{graphicx}

\newcommand{\Real}{\mathbb{R}}
\newcommand{\Realn}{\mathbb{R}^n}

\newcommand{\Toron}{\mathbb{T}^n}

\newcommand{\Zn}{\mathbb{Z}^n}
\newcommand{\N}{\mathbb{N}}

\newcommand{\PV}{\operatorname{P.V.}}
\newcommand{\Dom}{\operatorname{Dom}}

\newcommand{\s}{\sigma}
\renewcommand{\d}{\mathrm{dist}}

\newtheorem{thm}{Theorem}[section]
\newtheorem{ThA}{Theorem}

\newtheorem{lem}[thm]{Lemma}

\theoremstyle{definition}

\newtheorem{rem}[thm]{Remark}

\numberwithin{equation}{section}

\author[L. Roncal]{Luz Roncal}
\address{Departamento de Matem\'aticas y Computaci\'on\\
Universidad de La Rioja\\
26004 Logro\~no, Spain}
\email{luz.roncal@unirioja.es}

\author[P. R. Stinga]{Pablo Ra\'ul Stinga}
\address{Department of Mathematics\\
The University of Texas at Austin\\
1 University Station C1200\\
78712-1202 Austin, TX\\
United States of America}
\email{stinga@math.utexas.edu}

\thanks{The first author was partially supported by grant MTM2012-36732-C03-02 from Spanish Government. The second author was partially supported by MTM2011-28149-C02-01 from Spanish Government}



\begin{document}

\title[Transference of fractional Laplacian regularity]{Transference of fractional Laplacian regularity}

\begin{abstract}
In this note we show how to obtain regularity estimates for the fractional Laplacian on the multidimensional
torus $\Toron$ from the fractional Laplacian on $\Realn$. Though at first glance this may seem
quite natural, it must be carefully precised. A reason for that is the simple fact that $L^2$ functions on the torus
can not be identified with $L^2$ functions on $\Realn$. The transference is achieved through a formula
that holds in the distributional sense.
Such an identity allows us to transfer Harnack inequalities,
to relate the extension problems, and to obtain pointwise formulas and H\"older regularity estimates.
\end{abstract}

\maketitle

\section{The transference formula}

For $0<\sigma<1$ and $u:\Realn\to\Real$, the fractional Laplacian of order $\sigma$
 in $\Realn$ is defined using the Fourier transform as
$$(-\Delta_{\Realn})^\sigma u(x)=\int_{\Realn}|\xi|^{2\sigma}\widehat{u}(\xi)e^{i x
\cdot\xi}\,d\xi,\quad x\in\Realn.$$
Similarly, the fractional Laplacian on $\Toron\equiv\Realn/(2\pi\Zn)$ is defined via the
multiple Fourier series
$$(-\Delta_{\Toron})^\sigma v(z)=\sum_{k\in \Zn}|k|^{2\sigma}c_k(v)e^{ik\cdot z},$$
where $c_k(v)$ is the Fourier coefficient of $v:\Toron\to\Real$. In our notation,
the point $(e^{iz_1},\ldots,e^{iz_n})\in\Toron$ is uniquely identified with $z=(z_1,\ldots,z_n)\in Q_n
\coloneqq(-\pi,\pi]^n$, so $v(z)$ in fact means $v(e^{iz_1},\ldots,e^{iz_n})$.
In order to avoid a rather cumbersome notation, we will just write $z\in \Toron$.

It is clear that the fractional Laplacian on
$\Realn$ does not preserve the Schwartz class $\mathcal{S}$. Instead,
$$(-\Delta_{\Realn})^\sigma:\mathcal{S}\to
\mathcal{S}_\s\coloneqq\{\varphi\in C^\infty(\Realn):(1+|x|^2)^{\frac{n+2\s}{2}}D^{k}\varphi(x)
\in L^\infty(\Realn),k\in\N_0\},$$
see \cite[pp.~72--73]{Silvestre}.
Observe that $\mathcal{S}\subset \mathcal{S}_{\sigma}$.
Then the symmetry of the fractional Laplacian allows us
to define $(-\Delta_{\Realn})^{\s}$ for $u$ in the dual
space $\mathcal{S}_\s'$. For locally integrable \textit{functions} $u$ in $\mathcal{S}_\s'$ we let
$$\langle (-\Delta_{\Realn})^\s u,\varphi\rangle_{\mathcal{S}_\s}\coloneqq
\int_{\Realn}u(x)(-\Delta_{\Realn})^\s \varphi(x)\,dx,
\quad \varphi\in\mathcal{S}.$$
Certainly, the integral above is absolutely convergent when (see also \cite{Silvestre})
$$u\in L_\sigma\coloneqq L^1(\Realn,(1+|x|^2)^{-\frac{n+2\sigma}{2}}\,dx).$$

The situation with the fractional Laplacian on the torus is different than the $\Realn$ case.
We first notice that $(-\Delta_{\Toron})^\s$ preserves the
class of smooth functions on $\Toron$. By symmetry we
are able to define this operator for any \textit{function} $v$ that is
a periodic distribution on $\Toron$. Indeed, we let
$$\langle (-\Delta_{\Toron})^\s v,\phi\rangle_{C^\infty(\Toron)}
\coloneqq \int_{\Toron}v(z)(-\Delta_{\Toron})^\s\phi(z)\,dz,\quad
\phi\in C^{\infty}(\Toron).$$

To relate both fractional Laplacians we define two operators.
For a function $v$ on $\Toron$ we define its \textit{repetition} $Rv:\Realn\to\Real$ by
$$(Rv)(x)=\sum_{k\in\Zn}v(x-2\pi k)\chi_{Q_n}(x-2\pi k), \quad x\in \Realn.$$
This is nothing but the $Q_n$-periodic function on $\Realn$ that coincides with $v$ on $\Toron$. Here
$\Toron$ is identified with $Q_n$ as explained above.
For a function $u:\Realn\to\Real$ we define its \textit{periodization}
as the function $p_{\Sigma}u:\Toron\rightarrow\Real$ given (formally) by
\begin{equation}
\label{eq:periodization}
(p_{\Sigma}u)(z)=\sum_{k\in\Zn}u(z+2\pi k), \quad z\in \Toron.
\end{equation}
\begin{ThA}[Transference formula]\label{thm:transferencia}
Let $v$ be a function on the torus such that
\begin{equation}
\label{eq:ImportantCondition}
\sum_{k\in \Zn\setminus\{0\}}|c_k(v)|\frac{e^{-|k|^2}}{|k|}<\infty.
\end{equation}
Then its repetition $Rv$ is a function in $L_{\s}$ which defines a distribution in $\mathcal{S}_{\sigma}'$ and such that
\begin{equation}\label{identidad}
\int_{\Realn}(Rv)(-\Delta_{\Realn})^\s\varphi\,dx=\int_{\Toron}v(-\Delta_{\Toron})^\s(p_\Sigma\varphi)\,dz,\quad
\varphi\in\mathcal{S}.
\end{equation}
In other words, when evaluated in periodizations of Schwartz functions,
the periodic distribution $(-\Delta_{\Toron})^\s v$ coincides with
the distributional fractional Laplacian on $\Realn$ of its repetition $Rv$.
\end{ThA}

\begin{proof}
We first check that $Rv\in L_\s$. Let us compute
\begin{equation}\label{dale}
\begin{aligned}
\int_{\Realn}\frac{|(Rv)(x)|}{(1+|x|^2)^{\frac{n+2\sigma}{2}}}\,dx&=\sum_{k\in\Zn}\int_{Q_n}\frac{|(Rv)(x+2k\pi)|}{(1+|x+2k\pi|^2)^{\frac{n+2\sigma}{2}}}\,dx\\
&=\int_{\Toron}|v(z)|p_{\Sigma}((1+|\cdot|^2)^{-\frac{n+2\sigma}{2}})(z)\,dz.
\end{aligned}
\end{equation}
Since $(1+|\cdot|^2)^{-\frac{n+2\sigma}{2}}$ is integrable,
then $[p_{\Sigma}(1+|\cdot|^2)^{-\frac{n+2\sigma}{2}}]$ is integrable (see \cite[Chapter~VII]{Stein-Weiss}). Its Fourier coefficient
can be computed as follows:
\begin{align*}
	\mathcal{F}[(1+|\cdot|^2)^{-\frac{n+2\sigma}{2}}](k) &=
	 \mathcal{F}^{-1}(\mathcal{F}(I-\Delta_{\Realn})^{-\frac{n+2\s}{2}})(k) \\
	 &=\frac{1}{\Gamma(\frac{n+2\sigma}{2})}\int_0^{\infty}e^{-t}
	 \frac{e^{-|k|^2/(4t)}}{(4\pi t)^{n/2}}\,\frac{dt}{t^{1-\frac{n+2\sigma}{2}}} \\
	 &= \frac{|k|^{2\s}}{(4\pi)^{n/2}4^\sigma\Gamma(\frac{n+2\s}{2})}\int_0^\infty e^{-|k|^2/(4r)}e^{-r}\,
	 \frac{dr}{r^{1+\s}} = c_{n,\s}K_{\sigma}(|k|^2).
\end{align*}
Here $K_\s(z)$ is the modified Bessel function of the third kind (see \cite[p.~119]{Lebedev}).
A well known asymptotic formula gives that $K_\s(|k|^2)\sim |k|^{-1}e^{-|k|^2}$,
as $|k|\to\infty$. Hence, by Parseval's identity on $\Toron$ and the hypothesis, from \eqref{dale} we get
$$\int_{\Realn}\frac{|(Rv)(x)|}{(1+|x|^2)^{\frac{n+2\sigma}{2}}}\,dx =
	c_{n,\s}\sum_{k\in\Zn}|c_k(v)|K_\s(|k|^2) 
	\leq C_{n,\s}\sum_{k\in\Zn\setminus\{0\}}|c_k(v)|\frac{e^{-|k|^2}}{|k|}<\infty.$$
Thus, $Rv\in L_\s$ and the left hand side of \eqref{identidad} is absolutely convergent.

Again, $p_{\Sigma}\varphi$ is integrable on $\Toron$ and
$c_k(p_{\Sigma}\varphi)=\widehat{\varphi}(k)$, for each $k\in\Zn$.
Moreover, since $\varphi$ and $\widehat{\varphi}$ decay at infinity as $|x|^{-n-\delta}$, $\delta>0$, we have
\begin{equation}\label{serie de la periodizada}
(p_{\Sigma}\varphi)(z)=\sum_{k\in\Zn}\widehat{\varphi}(k)e^{i k\cdot z},
\end{equation}
where the series converges absolutely, see 
\cite[Chapter~VII]{Stein-Weiss}. From here, using the properties of the Fourier
transform, it readily follows that $p_\Sigma \varphi$ is a smooth function on the torus.
Hence $(-\Delta_{\Toron})^{\sigma}(p_{\Sigma}\varphi)$ is smooth too and the 
right hand side of \eqref{identidad} is absolutely convergent.

Before proving \eqref{identidad} we compute the periodization of $(-\Delta_{\Realn})^\s\varphi$. Since $\varphi$ is in the Schwartz class,
both $(-\Delta_{\Realn})^\s\varphi$ and its Fourier transform decay as $|x|^{-(n+2\s)}$ at infinity.
Therefore, by \eqref{serie de la periodizada},
\begin{align*}
	\left[p_{\Sigma}(-\Delta_{\Realn})^{\sigma}\varphi\right](z) &= \sum_{k\in \Zn}
	\widehat{(-\Delta_{\Realn})^{\sigma}\varphi}(k)e^{i k\cdot z}
	= \sum_{k\in \Zn}|k|^{2\sigma}\widehat{\varphi}(k)e^{i k \cdot z} \\
	 &=\sum_{k\in \Zn}|k|^{2\sigma}c_k(p_{\Sigma}\varphi)e^{i k\cdot z}
	 =(-\Delta_{\Toron})^{\sigma}(p_{\Sigma}\varphi)(z),
\end{align*}
for each $z\in \Toron$. With this, we readily obtain
\begin{align*}
	\int_{\Realn}(Rv)(-\Delta_{\Realn})^{\sigma}\varphi\,dx
		&= \int_{\Realn}\Bigg[\sum_{k\in\Zn} v(x-2\pi k)\chi_{Q_n}(x-2\pi k)\Bigg](-\Delta_{\Realn})^{\sigma}\varphi(x)\,dx \\
		&= \sum_{k\in \Zn}\int_{Q_n+2\pi k}v(x-2\pi k)(-\Delta_{\Realn})^{\sigma}\varphi(x)\,dx \\
		&= \sum_{k\in \Zn}\int_{Q_n}v(z)(-\Delta_{\Realn})^{\sigma}\varphi(z+2\pi k)\,dz \\
		&=\int_{\Toron}v(z)\left[p_{\Sigma}(-\Delta_{\Realn})^{\sigma}\varphi\right](z)\,dz
		=\int_{\Toron}v(-\Delta_{\Toron})^{\sigma}(p_{\Sigma}\varphi)\,dz.
\end{align*}
Notice that the integration on the torus with respect to the Haar measure is just the integration over $Q_n$
with respect to the Lebesgue measure, so the previous to last equality is true.
\end{proof}

\begin{rem}
Formula \eqref{identidad} is certainly valid for functions $v\in L^p(\Toron)$, $1\le p\le \infty$.
Indeed, $v\in L^1(\Toron)$ and, by the Riemann--Lebesgue Lemma, $c_k(v)\to 0$ as $|k|\to\infty$, thus
\eqref{eq:ImportantCondition} holds.
Observe that condition \eqref{eq:ImportantCondition} also holds whenever $\sum_{k\in \Zn\setminus\{0\}}|k|^{2\sigma}|c_k(v)|^2<\infty$, that is, when $v$ is in the Sobolev space $H^{\sigma}=\Dom((-\Delta_{\Toron})^{\sigma/2})$.
\end{rem}

\section{Applications}

\subsection{Harnack inequalities}

Interior and boundary Harnack estimates for the fractional Laplacian on the torus now follow from
the transference formula in Theorem \ref{thm:transferencia}.

\begin{thm}[Interior Harnack inequality]\label{Thm:interior Harnack}
Let $\mathcal{O}\subseteq\Toron$ be an open set. For any compact subset $\mathcal{K}\subset \mathcal{O}$,
there exists a constant $C>0$, that depends only on $n$, $\sigma$ and $\mathcal{K}$, such that
$$\sup_{\mathcal{K}}v\leq C\inf_{\mathcal{K}}v,$$
for all solutions $v\in\Dom((-\Delta_{\Toron})^\sigma)$ to
$$\begin{cases}
(-\Delta_{\Toron})^{\sigma}v=0,&\hbox{in}~\mathcal{O},\\
v\geq0,&\hbox{on}~\Toron.
\end{cases}$$
\end{thm}

\begin{proof}
For $v$ as in the hypothesis, its repetition $Rv$ is a nonnegative function on $\Realn$ which
belongs to $L_\sigma$. We can identify $\mathcal{O}$ with an open subset $\tilde{\mathcal{O}}\subset Q_n$.
Take any smooth function $\varphi$ with compact support in $\tilde{\mathcal{O}}$.
Then $p_\Sigma\varphi$ is a smooth function on the torus supported in $\mathcal{O}$.
Now Theorem \ref{thm:transferencia} gives that
$\langle (-\Delta_{\Realn})^\s(Rv),\varphi\rangle_{\mathcal{S}_\s}
=\langle (-\Delta_{\Toron})^\s v,p_\Sigma\varphi\rangle_{C^\infty(\Toron)}=0$.
Hence $Rv$ is a nonnegative solution to $(-\Delta_{\Realn})^\s(Rv)=0$ in $\tilde{\mathcal{O}}$.
Then $Rv$ satisfies Harnack inequality (see \cite[Theorem~5.1]{Caffarelli-Silvestre}), and so does $v$.
\end{proof}

\begin{thm}[Boundary Harnack inequality]\label{Thm:boundary Harnack}
Let $0<\sigma<1$ and $v_1,v_2\in\Dom((-\Delta_{\Toron})^\sigma)$
be two nonnegative functions on $\Toron$. Suppose that
$(-\Delta_{\Toron})^{\sigma}v_j=0$ in $\mathcal{O}$,
for some open set $\mathcal{O}\subseteq\Toron$.
Let $z_0\in \partial \mathcal{O}$ and assume that $v_j=0$
for all $z\in B_r(z_0)\cap \mathcal{O}^c$, for some sufficiently small $r>0$.
Assume also that $\partial \mathcal{O}\cap B_r(z_0)$ is a Lipschitz graph in the direction of $z_1$.
Then, there is a constant $C$ depending only on $\mathcal{O}$, $z_0$, $r$, $n$ and $\sigma$,
but not on $v_1$ or $v_2$, such that
$$\sup_{\mathcal{O}\cap B_{r/2}(z_0)}\left(\frac{v_1}{v_2}\right)
\leq C\inf_{\mathcal{O}\cap B_{r/2}(z_0)}\left(\frac{v_1}{v_2}\right).$$
Moreover, $v_1/v_2$ is $\alpha$-H\"older continuous
in $\overline{\mathcal{O}\cap B_{r/2}(z_0)}$, for some universal $0<\alpha<1$.
\end{thm}

\begin{proof}
Again we have that $Rv_i$ is in $L_\sigma$. We
identify $\mathcal{O}$ with an open subset $\tilde{\mathcal{O}}\subset Q_n$.
Let us also identify $z_0\in \partial \mathcal{O}$ with $x_0\in \partial \tilde{\mathcal{O}}$.
Then the corresponding boundary portion $\partial\tilde{\mathcal{O}}\cap B_r(x_0)$ is
a Lipschitz graph in the $x_1$-direction.
Using the same argument as in the proof of Theorem \ref{Thm:interior Harnack},
it follows that $Rv_i$ are nonnegative solutions to $(-\Delta_{\Realn})^{\s}(Rv_i)=0$ in $\tilde{\mathcal{O}}$,
and $Rv_i=0$ in $B_r(x_0)\cap\tilde{\mathcal{O}}^c$.
Therefore, the boundary Harnack inequality holds
for $Rv_i$ (see \cite[Theorem 5.3]{Caffarelli-Silvestre}), and so does for $v_i$
by restricting $Rv_i$ to $\mathcal{O}$.  The H\"older continuity of $v_1/v_2$ follows from the H\"older
continuity for $(Rv_1)/(Rv_2)$.
\end{proof}

\subsection{Extension problem}

It is known that the Caffarelli--Silvestre extension problem characterization
is valid also for the fractional Laplacian on the torus,
see \cite{Stinga,Stinga-Torrea}, also \cite{Roncal-Stinga, Gale-Miana-Stinga}. Here we can derive it directly
from the Caffarelli--Silvestre result of $\Realn$ in \cite{Caffarelli-Silvestre} with the explicit constants
computed in \cite{Stinga-Torrea}. In the proof we are going
to need the following simple result.

\begin{lem}\label{lem:lema-de-bumps}
Let $\phi$ be a smooth function on $\Toron$. Then there exists a smooth function $\varphi$ with compact support
on $\Realn$ such that
$$\phi(z)=p_{\Sigma}\varphi(z),\quad\hbox{for}~z\in \Toron.$$
\end{lem}

\begin{proof}
It is easy to see that
there exists a smooth function $\psi$ with compact support on $\Realn$ such that
$\sum_{k\in \Zn}\psi(x+2\pi k)\equiv 1$, for all $x\in \Realn$. Indeed,
$\psi$ can be constructed as the convolution of the characteristic function of $Q_n$ with a smooth bump function that
has integral $1$.  Set $\varphi(x)=\psi(x)(R\phi)(x)$.
Clearly, $\varphi$ is smooth (see the proof of Theorem \ref{thm:transferencia})
and has compact support. Moreover,
\begin{align*}
	(p_{\Sigma}\varphi)(z) &= \sum_{k\in \Zn}\psi(z+2\pi k)(R\phi)(z+2\pi k) \\
		&=\sum_{k\in \Zn}\psi(z+2\pi k)\sum_{j\in\Zn}\phi(z+2\pi k+2\pi j)\chi_{Q_n}(z+2\pi k+2\pi j) \\
		&=\phi(z)\sum_{k\in\Zn}\psi(z+2\pi k)=\phi(z),\quad z\in\Toron.
\end{align*}
\end{proof}

\begin{thm}[Extension problem]
Let $v\in \Dom((-\Delta_{\Toron})^{\sigma})$.
Let $V=V(z,y)$ be the solution to the boundary value problem
\begin{equation}\label{CStoro}
\begin{cases}
\Delta_{\Toron}V+\frac{1-2\sigma}{y}V_y+V_{yy}=0,&\hbox{in}~\Toron\times(0,\infty),\\
V(z,0)=v(z),&\hbox{on}~\Toron.
\end{cases}
\end{equation}
Then, for $c_\sigma=\frac{\Gamma(1-\sigma)}{4^{\sigma-1/2}\Gamma(\sigma)}>0$, we have that
\begin{equation}\label{CS2}
-\lim_{y\to0^+}y^{1-2\sigma}V_y(z,y)=c_\sigma(-\Delta_{\Toron})^{\sigma}v(z),\quad z\in\Toron.
\end{equation}
\end{thm}

\begin{proof}
Consider $u=Rv\in L_\sigma$. Let $U$ be the solution to the extension problem for $u$:
\begin{equation*}
\begin{cases}
\Delta_{\Realn}U+\frac{1-2\sigma}{y}U_y+U_{yy}=0,&\hbox{in}~\Realn\times(0,\infty),\\
U(x,0)=u(x),&\hbox{on}~\Realn.
\end{cases}
\end{equation*}
From \cite{Caffarelli-Silvestre} we know that $U(x,y)=P^\sigma_y\ast u(x)$,
for a suitable Poisson kernel $P^\sigma_y(x)$. Using this Poisson formula and
analogous to the proof of Theorem \ref{thm:transferencia},
it can be checked that $U(z,y)=(v\ast(p_\Sigma P^\sigma_y))(z)$, $z\in \Toron$, where the convolution
is performed on $\Toron$. Then $U(z,y)$ is a solution to \eqref{CStoro}.
By uniqueness, it follows that $V(\cdot,y)=v\ast(p_\Sigma P^\sigma_y)$,
for each $y>0$. Moreover, by Theorem \ref{thm:transferencia}
and the Caffarelli--Silvestre extension result for the fractional Laplacian on $\Realn$
in \cite{Caffarelli-Silvestre},
\begin{align*}
	c_\sigma\int_{\Toron}v (-\Delta_{\Toron})^\sigma(p_\Sigma \varphi)\,dz &=c_\sigma\int_{\Realn}
	u(-\Delta_{\Realn})^\sigma \varphi\,dx
	=-\lim_{y\to0^+}\int_{\Realn} y^{1-2\sigma}U_y(x,y)\varphi(x)\,dx \\
	&=-\lim_{y\to0^+}\sum_{k\in \Zn}\int_{Q_n} y^{1-2\sigma}U_y(z+2\pi k,y)\varphi(z+2\pi k)\,dz \\
	&=-\lim_{y\to0^+}\int_{Q_n} y^{1-2\sigma}U_y(z,y)(p_{\Sigma}\varphi)(z)\,dz \\
	& = -\lim_{y\to0^+}\int_{\Toron}y^{1-2\sigma}V_y(z,y)(p_\Sigma\varphi)(z)\,dz.
\end{align*}
Now \eqref{CS2} follows because any smooth function $\phi$ on the torus can be expressed
as $p_\Sigma\varphi$, for some $\varphi\in\mathcal{S}$, see Lemma \ref{lem:lema-de-bumps}.
\end{proof}

\subsection{Pointwise formula}

Let $0<\alpha\le1$ and $k\in\mathbb{N}_0$.
A continuous real function $v$ defined on $\Toron$ belongs to the H\"{o}lder space
$C^{k,\alpha}(\Toron)$, if $v\in C^k(\Toron)$ and
$$[D^\gamma v]_{C^\alpha(\Toron)}:=
\sup_{\begin{subarray}{c}x,y\in\Toron\\x\neq y\end{subarray}}
\frac{|D^\gamma v(x)-D^\gamma v(y)|}{\d(x,y)^{\alpha}}<\infty,$$
for each multi-index $\gamma\in \mathbb{N}_0^n$ such that $|\gamma|=k$.
Here $\d(x,y)$ is the geodesic distance from $x$ to $y$ on $\Toron$.
We define the norm in the spaces $C^{k,\alpha}(\Toron)$ as usual.

\begin{thm}[Pointwise formula]
\label{thm:puntual-toro}
Let $v\in C^{0,2\sigma+\varepsilon}(\Toron)$ if $0<\sigma<1/2$
(or $v\in C^{1,2\sigma+\varepsilon-1}(\Toron)$ if $1/2\leq\sigma<1$).
Then $(-\Delta_{\Toron})^{\sigma}v$ coincides with the continuous function on $\Toron$
given by
\begin{align*}
(-\Delta_{\Toron})^{\sigma}v(x) &= \PV\int_{\Toron}(v(x)-v(z))K_{\sigma}(x-z)\,dz\\
&=\lim_{\delta\to0^+}\int_{|x-z|>\delta,z\in\Toron}(v(x)-v(z))K_{\sigma}(x-z)\,dz,\quad x\in\Toron,
\end{align*}
where, for $x\in\Toron$, $x\neq0$,
$$K_{\sigma}(x)
=\frac{2^\sigma\Gamma\left(\frac{n+\sigma}{2}\right)}{|\Gamma(-\sigma/2)|\pi^{n/2}}
\sum_{k\in\Zn}\frac{1}{|x+2\pi k|^{n+2\sigma}}.$$
In the case $0<\sigma<1/2$ the integral above is in fact absolutely convergent.
\end{thm}

One may think that $K_\sigma$ is just the periodization of the kernel of the fractional Laplacian on $\Realn$.
In fact, formally, $K_\sigma(x)=c_{n,\sigma}p_\Sigma(|x|^{-(n+2\sigma)})$. But, since
$|x|^{-(n+2\sigma)}$ is not integrable on $\Realn$, this formal identity makes no sense. 

\begin{proof}[Proof of Theorem \ref{thm:puntual-toro}]
Notice that $K_\sigma(x)$ is well defined for $x\neq0$. Indeed, if $k\neq0$, then
for $x\in\Toron$ we have $|\pi k|\leq c_n|x-2k\pi|$, so
$$0\leq K_\sigma(x)\leq C_{n,\sigma}\left[\frac{1}{|x|^{n+2\sigma}}+\sum_{k\in\Zn\setminus\{0\}}
\frac{1}{|\pi k|^{n+2\sigma}}\right],\quad x\neq0,$$
and the series is absolutely convergent.
We have to prove that
\begin{equation}\label{eq:puntual-debil}
\langle(-\Delta_{\Toron})^{\sigma}v,
\phi\rangle_{C^{\infty}(\Toron)}=\int_{\Toron}h(x)\phi(x)\,dx,\quad\hbox{for any}~\phi\in C^\infty(\Toron),
\end{equation}
where the continuous function $h$ is given by
\[
h(x)=\PV\int_{\Toron}(v(x)-v(z))K_{\sigma}(x-z)\,dz.
\]
Let $u=Rv$. Then $u$ is bounded and it belongs to $C^{0,2\sigma+\varepsilon}(\Realn)$
(or to $C^{1,2\sigma+\varepsilon-1}(\Realn)$), so $(-\Delta_{\Realn})^{\sigma}u$ is a continuous function
on $\Realn$ (see \cite[Proposition 2.4]{Silvestre}) and
\begin{equation}
\begin{aligned}
\label{eq:puntual-repeticion}
(-\Delta_{\Realn})^{\sigma}u(x)&=c_{n,\sigma}\PV\int_{\Realn}
\frac{u(x)-u(y)}{|x-y|^{n+2\sigma}}\,dy\\
&=c_{n,\sigma}\PV\sum_{k\in \Zn}\int_{Q_n}
\frac{u(x)-u(z-2\pi k)}{|x-z+2\pi k|^{n+2\sigma}}\,dz\\
&=c_{n,\sigma}\PV\sum_{k\in \Zn}\int_{Q_n}
\frac{v(x)-v(z)}{|x-z+2\pi k|^{n+2\sigma}}\,dz\\
&=c_{n,\sigma}\PV\int_{\Toron}(v(x)-v(z))K_{\sigma}(x-z)\,dz=h(x).
\end{aligned}
\end{equation}
With this we conclude that $h$ is a continuous function on $\Toron$. Observe that
$(-\Delta_{\Realn})^{\sigma}u$ is a $Q_n$-periodic function.
To establish \eqref{eq:puntual-debil}, let $\phi$ be any smooth function on the torus.
By Lemma \ref{lem:lema-de-bumps}, there exists $\varphi\in\mathcal{S}$ such that $\phi(z)=p_{\Sigma}\varphi(z)$, $z\in \Toron$.
Then, by Theorem \ref{thm:transferencia} and \eqref{eq:puntual-repeticion},
\begin{align*}
\langle(-\Delta_{\Toron})^{\sigma}v,\phi\rangle_{C^{\infty}(\Toron)}
&=\langle(-\Delta_{\Toron})^{\sigma}v,
p_{\Sigma}\varphi\rangle_{C^{\infty}(\Toron)}=\langle(-\Delta_{\Realn})^{\sigma}u,
\varphi\rangle_{\mathcal{S}_{\sigma}}\\
&=\int_{\Realn}(-\Delta_{\Realn})^{\sigma}u(x)
\varphi(x)\,dx\\
&=\sum_{k\in \Zn}\int_{Q_n}(-\Delta_{\Realn})^{\sigma}u(x+2\pi k)\varphi(x+2\pi k)\,dx\\
&=\int_{Q_n}(-\Delta_{\Realn})^{\sigma}u(x)(p_{\Sigma}\varphi)(x)\,dx=\int_{\Toron}h(x)\phi(x)\,dx.
\end{align*}
\end{proof}

\subsection{H\"older regularity}

H\"older estimates follow directly from our transference formula and the known results for
the fractional Laplacian on $\Realn$.

\begin{thm}[H\"older estimates]
Take $\alpha\in (0,1]$.
\begin{enumerate}
\item[$(1)$] Let $v\in C^{0,\alpha}(\Toron)$ and $0<2\sigma<\alpha$. Then $(-\Delta_{\Toron})^{\sigma}v\in C^{0,\alpha-2\sigma}(\Toron)$ and
\[
\|(-\Delta_{\Toron})^{\sigma}v\|_{ C^{0,\alpha-2\sigma}(\Toron)}\le C\|v\|_{ C^{0,\alpha}(\Toron)}.
\]
\item[$(2)$] Let $v\in C^{1,\alpha}(\Toron)$ and $0<2\sigma<\alpha$. Then $(-\Delta_{\Toron})^{\sigma}v\in C^{1,\alpha-2\sigma}(\Toron)$ and
\[
\|(-\Delta_{\Toron})^{\sigma}v\|_{C^{1,\alpha-2\sigma}(\Toron)}\leq C\|v\|_{C^{1,\alpha}(\Toron)}.
\]
   \item[$(3)$] Let $v\in C^{1,\alpha}(\Toron)$ and $2\sigma\geq\alpha$, with $\alpha-2\sigma+1\neq0$.
   Then $(-\Delta_{\Toron})^{\sigma}v\in C^{0,\alpha-2\sigma+1}(\Toron)$ and
\[
\|(-\Delta_{\Toron})^{\sigma}v\|_{C^{0,\alpha-2\sigma+1}(\Toron)}\leq C\|v\|_{C^{1,\alpha}(\Toron)}.
\]
\item[$(4)$] Let $v\in C^{k,\alpha}(\Toron)$ and assume that $k+\alpha-2\sigma$ is not an integer.
Then $(-\Delta_{\Toron})^\sigma v\in C^{l,\beta}(\Toron)$, where $l$ is the integer part of $k+\alpha-2\sigma$
and $\beta=k+\alpha-2\sigma-l$, and
$$\|(-\Delta_{\Toron})^\sigma v\|_{C^{l,\beta}(\Toron)}\leq C\|v\|_{C^{k,\alpha}(\Toron)}.$$
\end{enumerate}
\end{thm}

\begin{proof}
For (1), by Theorem \ref{thm:puntual-toro} and \cite[Proposition 2.5]{Silvestre} we readily get,
$$\|(-\Delta_{\Toron})^{\sigma}v\|_{ C^{0,\alpha-2\sigma}(\Toron)} = \|(-\Delta_{\Toron})^{\sigma}(Rv)\|_{ C^{0,\alpha-2\sigma}(\Realn)}
\le C\|Rv\|_{C^{0,\alpha}(\Realn)}=
C\|v\|_{C^{0,\alpha}(\Toron)}.$$
Parts (2), (3) and (4) follow analogously by using Theorem \ref{thm:transferencia},
 Theorem \ref{thm:puntual-toro} and the known results for $\Realn$ \cite[Proposition~2.6,
 ~Proposition~2.7]{Silvestre}.
\end{proof}

\noindent\textbf{Acknowledgement.} We thank Luis Caffarelli and Jos\'e L. Torrea for delightful
and pleasant discussions about the results of this work.



\end{document}